\documentclass[12pt]{amsart}
\usepackage{amssymb,amsmath,amsthm,enumerate,comment,booktabs,xcolor}
\usepackage[mathscr]{euscript}
\usepackage[a4paper]{geometry}
\usepackage[T1]{fontenc}
\usepackage[utf8]{inputenc}

\theoremstyle{plain}
\newtheorem{theorem}{Theorem}[section]
\newtheorem{proposition}[theorem]{Proposition}

\theoremstyle{definition}
\newtheorem{remark}[theorem]{Remark}
\newtheorem{definition}[theorem]{Definition}

\numberwithin{equation}{section}

\def\C {\mathbb{C}}

\usepackage{mathtools}
\DeclarePairedDelimiter\floor{\lfloor}{\rfloor}

\def \au {\rm}
\def \ti {\it}
\def \jou {\rm}
\def \bk {\it}
\def \no#1#2#3 {{\bf #1} (#3), #2.}
\def \eds#1#2#3 {#1, #2, #3.}
\newcommand{\T}{(T(t))_{t\ge 0}}

\DeclareMathOperator*{\re}{Re}

\begin{document}

\title[Growth rates for operator semigroups]
{Unified growth rates for operator semigroups\\ under generalized Kreiss conditions}

\author[F. Dell'Oro, A. Farina, V. Pata]
{Filippo Dell'Oro$^*$, Alberto Farina$^\dag$ and Vittorino Pata$^*$}

\address{$^*$Politecnico di Milano - Dipartimento di Matematica
\newline\indent
Via Bonardi 9, 20133 Milano, Italy}
\email{filippo.delloro@polimi.it}
\email{vittorino.pata@polimi.it}

\address{$^\dag$Université de Picardie Jules Verne - LAMFA
\newline\indent
CNRS UMR 7352, 33, Rue Saint-Leu 80039 Amiens, France}
\email{alberto.farina@u-picardie.fr}


\begin{abstract}
In this note we establish a unified growth rate for the operator norm of $C_0$-semigroups on Hilbert spaces
whose generators satisfy the generalized Kreiss resolvent condition.
Our bound contains and improves several known estimates in the literature. In particular, it
captures the transition between different super-linear growth behaviors.
\end{abstract}

\maketitle

\section{Introduction}

\noindent
Let $\T$ be a $C_0$-semigroup on a complex Hilbert space $H$ with infinitesimal generator $A$,
and let $\rho(A)$ be the resolvent set of $A$. We denote by
$$R(\lambda,A)=(\lambda- A)^{-1}$$
the resolvent operator at $\lambda\in\rho(A)$.

\begin{definition}
We say that $A$ satisfies the \emph{generalized Kreiss condition} if the open right half-plane
$\C^{+} = \{ \lambda \in \C : \re \lambda > 0\}$
is contained in $\rho(A)$ and there exists a non-decreasing function $g:(0,\infty)\to (0,\infty)$
such that
\begin{equation}
\label{gen-kreiss}
\|R(\lambda,A)\|\leq g \Big(\frac{1}{\re \lambda}\Big),\quad\forall \lambda \in \C^{+}.
\end{equation}
\end{definition}

\begin{remark}
In fact, as our result will involve only the large-time behavior of $g$,
it is enough to assume~\eqref{gen-kreiss} for $\re\lambda<s_0$, with $s_0>0$ large enough.
For $\re\lambda\geq s_0$, the usual
Hille-Yosida resolvent estimate allows one to recover the global formulation
after a harmless redefinition of $g$ on $(0,s_0]$.
\end{remark}

The particular case $g(t)= C t\,$ for some constant $C\geq 1$ corresponds to the
\emph{standard Kreiss condition},
first introduced by Kreiss \cite{kreiss-paper} in the finite-dimensional case.
By adapting a technique devised in \cite{BoMu,CCEL}
for discrete semigroups, it has been proved in \cite{ARN} that within the standard Kreiss condition
the following growth rate holds:
\begin{equation}
\label{arnold-est}
\|T (t)\|=O\left(\frac{t}{\sqrt{\log t}}\right),
\end{equation}
see also \cite{EZ,RO}.

\begin{remark}
The standard Kreiss condition is significant only if $C>1$. Indeed, when $C=1$ the semigroup is
a contraction semigroup (i.e., $\|T (t)\|\leq 1$),
as a consequence of the Hille-Yosida theorem; see, e.g., \cite[Theorem II.3.5]{ENG}.
\end{remark}

As shown in \cite{ROVE}, the validity of~\eqref{gen-kreiss} implies that
\begin{equation}
\label{gen-kreiss-est}
\|T (t)\|=O(g(t)),
\end{equation}
see also \cite{Bou,EZ2,HEL,HEL2,HEL3,RO,ROVE}.

\begin{remark}
It is well known that within the generalized Kreiss condition (in a Hilbert space setting) the semigroup growth rate is
always subexponential;
see, e.g., \cite{DE,ROVE}. Accordingly, it makes no sense
to consider~\eqref{gen-kreiss}
with a fast growing $g$.
 
\end{remark}
Despite its greater generality, estimate~\eqref{gen-kreiss-est} is not sharp:\
this is immediately seen when $g(t)= C t\,$, where~\eqref{arnold-est} provides
a better bound.
More recently, it has been proved in \cite{DE} that
if~\eqref{gen-kreiss} is satisfied and the function $h(t)=g(t)/t$ is non-decreasing, then
\begin{equation}
\label{mia-est}
\|T (t)\|= O\left(g(t)\frac{h(t)}{\sqrt{\log t}}\right).
\end{equation}
In the case of the standard Kreiss condition, we recover exactly~\eqref{arnold-est}.
In fact, \eqref{mia-est} is sharper than \eqref{gen-kreiss-est} whenever $g(t)=o(t\sqrt{\log t}\,)$,
but can be worse otherwise.
More details concerning the Kreiss conditions and their applications may be found for instance in \cite{ARN2,EZ2,ROVE,ROVE2}.

\subsection*{Notation} Given two positive functions $f_1,f_2$ defined in a neighborhood of $+\infty$,
we write $f_1(t) = O(f_2(t))$ to mean  $f_1(t) \leq c f_2(t)$ for some constant $c>0$ and
every $t > 0$ sufficiently large, $f_1(t) = o(f_2(t))$
to mean $f_1(t)/f_2(t)\to 0$ as $t\to\infty$, and $f_1(t) \sim f_2(t)$ to mean $f_1(t)/f_2(t)\to 1$ as $t\to\infty$.


\section{The Result}

\noindent
The aim of this note is to provide a unified growth rate that captures
the classical bound \eqref{gen-kreiss-est} and, at the same time, recovers
or improves the modified
bound \eqref{mia-est} for fine super-linear growth behaviors.
To this end, let us define the integral function
$$
\Lambda(t) = \int_{2}^t \frac{s}{[g(s)]^2}\, ds,\quad t>2.
$$
In this definition, the particular choice of the lower limit $2$ is irrelevant for large-time
estimates and may be replaced by any positive number.
Our main result, proved in the last Section~\ref{sec:proof}, reads as follows.

\begin{theorem}
\label{thm-kreiss}
Let $\T$ be a $C_0$-semigroup on a complex Hilbert space $H$ with infinitesimal generator $A$
satisfying the generalized Kreiss condition~\eqref{gen-kreiss}.
Then
$$
\|T (t)\|= O\left(\frac{g(t)}{\sqrt{\Lambda(t)}}\right).
$$
\end{theorem}

Note that, unlike in~\cite{DE}, we do not
require the additional monotonicity assumption on the function $h(t)=g(t)/t$.
Theorem \ref{thm-kreiss} embodies both
estimates~\eqref{gen-kreiss-est} and~\eqref{mia-est}.

\begin{proposition}
\label{coro}
The following hold:
\begin{enumerate}
\item[\rm (i)] We always have
$$\frac{g(t)}{\sqrt{\Lambda(t)}}=O(g(t)).$$
In particular,
$$\frac{g(t)}{\sqrt{\Lambda(t)}}=o(g(t))$$
whenever $\Lambda(t)\to\infty$, i.e., whenever $t/[g(t)]^2$ is not integrable at infinity.
\vskip2mm
\item[\rm (ii)] If $h$ is non-decreasing we have
$$
\frac{g(t)}{\sqrt{\Lambda(t)}} = O\left(g(t)\frac{h(t)}{\sqrt{\log t}}\right).
$$
\end{enumerate}
\end{proposition}

\begin{proof}
As $\Lambda$ is increasing, point (i) is immediate.
To show point (ii), we exploit the monotonicity of $h$.
Then, for $2<s<t$, we write $g(s)\leq sh(t)$. Hence,
$$\Lambda(t)\geq \frac{1}{[h(t)]^2} \int_2^t\frac{1}{s}\,ds
\sim \frac{\log t}{[h(t)]^2},$$
yielding the desired conclusion.
\end{proof}

It is important to note that, when $\Lambda(t)\to\infty$, point (i) of the proposition above
improves the classical bound~\eqref{gen-kreiss-est}.
A precise quantification of such an improvement clearly depends
on the asymptotic behavior of $g$. In the next section, we will provide some examples in this direction.

\section{Some Examples}

\noindent
We now illustrate the effectiveness of Theorem \ref{thm-kreiss}
focusing on some selected examples.
These cases are by no means exhaustive, and further examples can be easily constructed along the same lines.

\subsection{Power super-linear growth}
If $g(t)=t^\alpha$ with $\alpha>1$ then $t/[g(t)]^2$ is integrable at infinity.
In this situation, we recover \eqref{gen-kreiss-est}
which is sharper than \eqref{mia-est}.

\begin{table}[htbp]
\centering
\renewcommand{\arraystretch}{2.1}
\setlength{\tabcolsep}{18pt}
\begin{tabular}{|c|c|c|}
\hline
\textbf{Bound (\ref{gen-kreiss-est})} & \textbf{Bound (\ref{mia-est})} & \textbf{Bound Thm.\ \ref{thm-kreiss}} \\
\hline\hline
$O(t^\alpha)$ & $O\left(\displaystyle\frac{t^{2\alpha-1}}{\sqrt{\log t}}\right)$ & ${O(t^\alpha)}$ \\[1ex]
\hline
\end{tabular}
\vspace{0.3cm}
\caption{$g(t)=t^\alpha$ with $\alpha > 1$.}
\end{table}

\subsection{Fine super-linear growth}
We consider functions $g$ growing slightly faster than a linear function, introducing logarithmic corrections.

\smallskip
\begin{enumerate}

\item[$\bullet$] If $g(t)=t\, (\log(t+1))^\beta$ with $\beta\in(0,\frac{1}{2})$
then
$$
\Lambda(t) \sim \frac{(\log t)^{1-2\beta}}{1-2\beta}.
$$
We recover \eqref{mia-est}, which is sharper than \eqref{gen-kreiss-est}.

\medskip
\item[$\bullet$]
If $g(t)=t \sqrt{\log(t+1)}$ then
$$
\Lambda(t) \sim \log \log t.
$$
Here we get the bound
$$
\|T (t)\|= O\left(\frac{t \sqrt{\log t}}{\sqrt{\log \log t}}\right),
$$
sharper than both \eqref{gen-kreiss-est} and \eqref{mia-est}.

\medskip
\item[$\bullet$]
If $g(t)=t\, (\log(t+1))^{\beta}$ with $\beta>\frac12$ then
$t/[g(t)]^2$ is integrable at infinity. In this situation we recover \eqref{gen-kreiss-est}
which is sharper than \eqref{mia-est}.
\end{enumerate}

\begin{table}[htbp]
\centering
\renewcommand{\arraystretch}{2.3}
\setlength{\tabcolsep}{16pt}
\begin{tabular}{|c|c|c|}
\hline
\textbf{Bound (\ref{gen-kreiss-est})} & \textbf{Bound (\ref{mia-est})} & \textbf{Bound Thm.\ \ref{thm-kreiss}} \\
\hline\hline
$O\big(t (\log t)^\beta\big)$\quad $\beta \in \big(0,\frac12\big)$
& $O\left(\displaystyle\frac{t (\log t)^{\beta}}{\sqrt{(\log t)^{1-2\beta}}}\right)$
& $O\left(\displaystyle\frac{t (\log t)^{\beta}}{\sqrt{(\log t)^{1-2\beta}}}\right)$ \\
\hline
$O(t \sqrt{\log t})$ & $O\big(t \sqrt{\log t}\big)$
& $O\left(\displaystyle\frac{t \sqrt{\log t}}{\sqrt{\log \log t}}\right)$ \\ [1.2ex]
\hline
$O\big(t (\log t)^\beta\big)$\quad $\beta >\frac12$
& $O\big(t (\log t)^{2\beta-\frac12}\big)$ & $O\big(t (\log t)^\beta\big)$ \\
\hline
\end{tabular}
\vspace{0.25cm}
\caption{$g(t)=t(\log(t+1))^\beta$ with $\beta>0$.}
\end{table}

\subsection{Finer super-linear growth}
We now refine the analysis by adding a double logarithmic correction.

\smallskip
\begin{enumerate}
\item[$\bullet$]
If $g(t)=t \sqrt{\log(t+1)} (\log \log(t+e))^{\gamma}$ with $\gamma\in(0,\frac{1}{2})$
then
$$
\Lambda(t) \sim \frac{(\log \log t)^{1-2\gamma}}{1-2\gamma}.
$$
We get the bound
$$
\|T (t)\|= O\left(\frac{t\, \sqrt{\log t}\,
(\log \log t)^\gamma}{\sqrt{(\log \log t)^{1-2\gamma}}}\right),
$$
sharper than both \eqref{gen-kreiss-est} and \eqref{mia-est}.

\medskip
\item[$\bullet$]
If $g(t)=t \sqrt{\log(t+1)} \sqrt{\log \log(t+e)}$
then
$$
\Lambda(t) \sim \log \log \log t,
$$
and we get the bound
$$
\|T (t)\|= O\left(\frac{t\, \sqrt{\log t}\, \sqrt{\log \log t}}{\sqrt{\log \log \log t}}\right),
$$
sharper than both \eqref{gen-kreiss-est} and \eqref{mia-est}.

\medskip
\item[$\bullet$]
If $g(t)=t \sqrt{\log(t+1)} (\log \log(t+e))^{\gamma}$ with $\gamma>\frac12$ then
$t/[g(t)]^2$ is integrable at infinity.
In this case we recover \eqref{gen-kreiss-est},
which is sharper than \eqref{mia-est}.
\end{enumerate}

\begin{table}[htbp]
\centering
\renewcommand{\arraystretch}{2.4}
\setlength{\tabcolsep}{3.5pt}
\begin{tabular}{|c|c|c|}
\hline
 \textbf{Bound (\ref{gen-kreiss-est})}
& \textbf{Bound (\ref{mia-est})} & \textbf{Bound Thm.\ \ref{thm-kreiss}} \\
\hline\hline
$O\big(t \sqrt{\log t} (\log \log t)^\gamma\big)$\quad $\gamma \in \big(0,\frac12\big)$
& $O\big(t \sqrt{\log t} (\log \log t)^{2\gamma}\big)$
& $O\left(\displaystyle\frac{t \sqrt{\log t} (\log \log t)^\gamma}{\sqrt{(\log \log t)^{1-2\gamma}}}\right)$ \\
\hline
$O\big(t \sqrt{\log t} \sqrt{\log \log t}\big)$
& $O\big(t \sqrt{\log t} (\log \log t)\big)$
& $O\left(\displaystyle\frac{t \sqrt{\log t} \sqrt{\log \log t}}{\sqrt{\log \log \log t}}\right)$ \\ [1.2ex]
\hline
$O\big(t \sqrt{\log t} (\log \log t)^\gamma\big)$\quad $\gamma>\frac12$
& $O\big(t \sqrt{\log t} (\log \log t)^{2\gamma}\big)$
& $O\big(t \sqrt{\log t} (\log \log t)^\gamma\big)$ \\
\hline
\end{tabular}
\vspace{0.25cm}
\caption{$g(t)=t \sqrt{\log(t+1)} (\log \log(t+e))^{\gamma}$ with $\gamma>0$.}
\end{table}

\subsection{Linear growth}
Finally, we consider the linear growth $g(t)=Ct$ for some constant $C\geq1$,
which corresponds to the standard Kreiss condition.
Here we have
$$
\Lambda(t) \sim \frac{\log t}{C^2},
$$
and we recover~\eqref{arnold-est} and~\eqref{mia-est},
sharper than \eqref{gen-kreiss-est}.

\begin{table}[htbp]
\centering
\renewcommand{\arraystretch}{2.1}
\setlength{\tabcolsep}{16pt}
\begin{tabular}{|c|c|c|}
\hline
\textbf{Bound (\ref{gen-kreiss-est})}
& \textbf{Bounds (\ref{arnold-est}) and (\ref{mia-est})} & \textbf{Bound Thm.\ \ref{thm-kreiss}} \\[1ex]
\hline\hline
 $O(t)$ & $\displaystyle O\left(\frac{t}{\sqrt{\log t}}\right)$
& $\displaystyle O\left(\frac{t}{\sqrt{\log t}}\right)$ \\[1ex]
\hline
\end{tabular}
\vspace{0.25cm}
\caption{$g(t)=Ct$ with $C\geq1$ (standard Kreiss condition).}
\end{table}

\section{Proof of Theorem \ref{thm-kreiss}}
\label{sec:proof}

\noindent
Let $x,y\in H$ be arbitrarily fixed vectors of unit norm.
As shown in the proof of \cite[Theorem 1.1]{DE}, the validity of the generalized Kreiss
condition \eqref{gen-kreiss} implies the following integral bounds
on $\T$ and the adjoint semigroup $(T^*(t))_{t\ge0}$
\begin{align}
\label{stima1}
\int_{0}^{t}\|T(\tau)x\|^2\, \mathrm{d}\tau \leq K[g(t)]^2,\\
\noalign{\vskip1mm}
\label{stima2}
\int_{0}^{t}\|T^*(\tau)y\|^2\, \mathrm{d}\tau\leq K[g(t)]^2,
\end{align}
for every $t>1$ and some constant $K>0$ independent of $t,x,y$\footnote{Estimates 
\eqref{stima1}-\eqref{stima2} in \cite{DE} appear written with $1+g(t)$ in place of $g(t)$.
Since $t>1$ and $g$ is non-decreasing, this is clearly inessential.}.
We emphasize that
such bounds are deduced by means 
of the Plancherel theorem
for square-integrable functions taking values in a Hilbert space;
see e.g.,\ \cite[Appendix~C]{ENG}.
Let now $t\geq2$ be arbitrarily fixed. For $1\leq a<b\leq t$ to be chosen later, we write
\begin{align*}
(b-a)|\langle T(t) x, y \rangle| &= \int_a^b |\langle T(t-\tau) x, T^*(\tau) y \rangle|\mathrm{d}\tau \\
&\leq
\Big(\int_0^b \| T^*(\tau) y\|^2 \mathrm{d}\tau  \Big)^{1/2}
\Big(\int_a^b \| T(t-\tau) x\|^2 \mathrm{d}\tau  \Big)^{1/2}.
\end{align*}
Applying \eqref{stima2} to the first integral, and making a change of variable in the second one,
we obtain
$$
(b-a)|\langle T(t) x, y \rangle|\leq \sqrt{K}\,g(b)
\bigg(\int_{t-b}^{t-a} \| T(\tau) x\|^2 \mathrm{d}\tau\bigg)^{1/2}.
$$
Taking the supremum over $\|y\|=1$ and squaring both sides, we arrive at
$$
\left(\frac{b-a}{g(b)}\right)^2\|T(t)x\|^2 \leq {K}\int_{t-b}^{t-a}\|T(\tau)x\|^2 d\tau.
$$
At this point, setting $N(t)= \floor*{\log_2 t}$, where $\floor*{\,\cdot\,}$ is the floor function,
we choose $a= 2^{n-1}$ and $b=2^{n}$ for $1\leq n\leq N(t)$. With this choice, the inequality above becomes
$$
\frac{1}{4}\left(\frac{2^n}{g(2^n)}\right)^2 \|T(t)x\|^2 \leq K\int_{t-2^n}^{t-2^{n-1}}\|T(\tau)x\|^2 d\tau.
$$
{Summing on $1\leq n\leq N(t)$} and using \eqref{stima1}, we find
\begin{align}
\label{almost}
\frac{1}{4}\sum_{n=1}^{N(t)}\left(\frac{2^n}{g(2^n)}\right)^2 \|T(t)x\|^2
&\leq K \sum_{n=1}^{N(t)} \int_{t-2^n}^{t-2^{n-1}}\|T(\tau)x\|^2 d\tau\\\nonumber
&\leq  K \int_{0}^{t}\|T(\tau)x\|^2 d\tau\\\noalign{\vskip3.3mm}\nonumber
&\leq K^2[g(t)]^2.
\end{align}
We now need a lower bound for the sum in the left-hand side.
Recalling that $g$ is non-decreasing, we estimate
$$
\int_{2^{n}}^{2^{n+1}}\frac{s}{[g(s)]^2} ds \leq \frac{1}{[g(2^{n})]^2}
\int_{2^{n}}^{2^{n+1}}s\, ds  =  \frac32 \left(\frac{2^n}{g(2^n)}\right)^2.
$$
Since $2^{N(t)+1}>t$, we infer that
$$
\sum_{n=1}^{N(t)}\left(\frac{2^n}{g(2^n)}\right)^2 \geq
\frac23 \sum_{n=1}^{N(t)} \int_{2^{n}}^{2^{n+1}}\frac{s}{[g(s)]^2} ds > \frac23 \Lambda(t).
$$
Plugging this inequality into~\eqref{almost}, we end up with
$$
\|T(t)x\|^2
\leq \frac{6 K^2[g(t)]^2}{\Lambda(t)}.
$$
Finally, taking the supremum over $\|x\|=1$, we are led to
$$
\|T(t)\|
\leq \frac{\sqrt{6} K g(t)}{\sqrt{\Lambda(t)}},
$$
and the desired conclusion follows.
\qed


\end{document}